\documentclass[reqno, 12pt, reqno]{amsart}

\usepackage{amsfonts, amsthm, amsmath, amssymb}
\usepackage{hyperref}
\hypersetup{colorlinks=false}

\usepackage[margin=1.5in]{geometry}

\RequirePackage{mathrsfs} \let\mathcal\mathscr
\numberwithin{equation}{section}

\renewcommand{\phi}{\varphi}
\renewcommand{\rho}{\varrho}

\newcommand{\0}{\mathbf{0}}
\newcommand{\PP}{\mathbb{P}}

\newcommand{\ZZ}{\mathbb{Z}}
\newcommand{\ZZp}{\mathbb{Z}_{\mathrm{prim}}}
\newcommand{\NN}{\mathbb{N}}
\newcommand{\QQ}{\mathbb{Q}}
\newcommand{\RR}{\mathbb{R}}

\renewcommand{\leq}{\leqslant}
\renewcommand{\le}{\leqslant}
\renewcommand{\geq}{\geqslant}

\newcommand{\x}{\mathbf{x}}
\newcommand{\y}{\mathbf{y}}
\renewcommand{\c}{\mathbf{c}}

\renewcommand{\u}{\mathbf{u}}
\newcommand{\z}{\mathbf{z}}

\newcommand{\ve}{\varepsilon}

\DeclareMathOperator{\meas}{meas}

\renewcommand{\le}{\left}
\newcommand{\ri}{\right}

\newtheorem{thm}{Theorem}[section]

\newtheorem{lem}[thm]{Lemma}

\theoremstyle{definition}

\numberwithin{equation}{section}

\begin{document}
\date{\today}

\title[Rational points on biquadratic hypersurfaces]{Counting  rational points on  \\
biquadratic hypersurfaces}
\author{T.D.\ Browning}

\address{IST Austria\\
Am Campus 1\\
3400 Klosterneuburg\\
Austria}
\email{tdb@ist.ac.at}

\author{L.Q. Hu}
\address{Dept.\ of Math.\ \\ Nanchang University\\
Nanchang\\ Jiangxi 330031\\ China}
\email{huliqun@ncu.edu.cn}

\thanks{2010  {\em Mathematics Subject Classification.} 11D45 (11G35, 11G50, 11P55, 14G25)}

\begin{abstract}
An asymptotic formula is established for the number of rational points of bounded anticanonical height which lie on a certain Zariski open subset of an arbitrary smooth biquadratic hypersurface in sufficiently many variables. The proof uses the Hardy--Littlewood circle method.
\end{abstract}

\maketitle

\thispagestyle{empty}
\setcounter{tocdepth}{1}
\tableofcontents

\section{Introduction}

Let $F(\mathbf{x};\mathbf{y})$ be a bihomogeneous form of bidegree $(d_1,d_2)$ defined over $\ZZ$, 
where  $\mathbf{x}=(x_1,\dots,x_n)$ and $\mathbf{y}=(y_1,\dots,y_n)$. 
This polynomial defines a biquadratic hypersurface 
$$
X:\quad  \{F(\mathbf{x};\mathbf{y})=0\} \subset \mathbb{P}^{n-1}\times \mathbb{P}^{n-1}.
$$
For a point $(x,y)\in X(\QQ)$, with $x=(x_1:\dots:x_n)$ and 
$y=(y_1:\dots:y_n)$ chosen so that 
 $\x,\y\in \ZZp^n$,
the anticanonical height is
$$
H(x,y)=|\mathbf{x}|^{n-d_1}|\mathbf{y}|^{n-d_2},
$$ 
where $|\cdot|:\RR^n\to \RR$ is the sup-norm. 
We assume that $X$ is a smooth Fano variety of dimension at least $3$
and that $X(\QQ)\neq \emptyset$.
Then, according to the Manin conjecture 
\cite{F}, there should exist a Zariski open subset $U\subset X$ and a constant $c> 0$ such that 
\begin{equation}\label{eq:manin}
\#\{(x,y)\in U(\QQ): H(x,y)\leq B\}\sim c B \log B,
\end{equation}
as $B\rightarrow \infty$. Furthermore, 
a conjectured value  for the constant  $c$ has been put forward by Peyre \cite{P}.

We now know that \eqref{eq:manin} is false in general. Recent work of Browning and Heath-Brown \cite{BBR} shows that 
for the smooth hypersurface
$$
\{x_1y_1^2+\dots+x_4y_4^2=0\} \subset \PP^3\times\PP^3,
$$
of bidegree $(1,2)$,
one needs to remove certain  ``thin sets'' of rational points in order to 
arrive at  the Manin--Peyre prediction. For larger values of $ n$, however,  we expect that 
it suffices to merely remove proper closed subvarieties, as allowed for in
\eqref{eq:manin}.

 When  $(d_1,d_2)=(1,1)$,  with $F(\mathbf{x};\mathbf{y})=x_1y_1+\dots+x_ny_n$, Robbiani \cite{R}  has confirmed \eqref{eq:manin} for all
$n\geq 4$.
This has been extended to $n\geq 3$ by Spencer \cite{spencer}.
For general bidegree $(d_1,d_2)$ the best result available is due to Schindler \cite{S2}, who has verified the  Manin--Peyre conjecture \eqref{eq:manin} for an appropriate open subset $U\subset X$,  provided that 
$$
n>3\cdot 2^{d_1+d_2}d_1d_2+1.
$$
(Here we have used \cite[Lemma 2.2]{S2} to note that $2n-\max\{\dim V_1^*, \dim V_2^*\}\geq n-1$ in \cite[Thm.~1.1]{S2}.)

The primary purpose of the present paper is to extend the range of $n$ in the special case of bidegree $(d_1,d_2)=(2,2)$. 
Let us denote
the left hand side  of \eqref{eq:manin}
 by  $N_U(B)$.
We write 
$F_\x$ for the quadratic form obtained from $F(\x;\y)$ when $\x$ is fixed and $y_1,\dots,y_n$ are  viewed as the  variables.
Note that  the determinant $\det(F_\x)$ of $F_\x$ depends only on $\x$. 
Similarly, we write 
$F_\y$ for the quadratic form in $x_1,\dots,x_n$ obtained from $F(\x;\y)$ by fixing $\y$.
We shall take $U=X\setminus Z$ in our work, where
\begin{equation}\label{eq:Z}
Z=\{(x,y)\in X: \det (F_{\mathbf{x}})\det (F_{\mathbf{y}})=0\}.
\end{equation}
We shall show that $U\neq \emptyset$ at the start of \S \ref{s:4}.
As consideration of the 
hypersurface 
$
x_1^2y_1^2+\dots +x_n^2y_n^2=0
$
shows, it is in fact necessary to work with the open subset $U$ that we have defined. Indeed, for this particular example, one sees that  rational points with  $x_1=0$ and $y_2=\dots=y_n=0$
contribute  $\gg B^{(n-1)/(n-2)}$ to $N_U(B)$. 

The following is our main result. 

\begin{thm}\label{t:1}
Let $X\subset \PP^{n-1}\times \PP^{n-1}$ be a smooth hypersurface of bidegree $(2,2)$, with
$n>35$.
 Let $U=X\setminus Z$, where $Z$ is given by \eqref{eq:Z}.
Then 
$$N_{U}(B)\sim c B\log B,$$
as $B\rightarrow \infty$, where 
 $c$  is the constant predicted by Peyre \cite{P}.
\end{thm}

By way of comparison, Schindler's result \cite{S2} yields the same conclusion for $n>193$. 
(In fact the latter result also  handles complete intersections in $\PP^{n_1-1}\times \PP^{n_2-1}$ when 
$\min\{n_1,  n_2\}$ is large enough.)

Our proof of Theorem \ref{t:1} uses the Hardy--Littlewood circle method and draws heavily on the strategies adopted in  the works of Browning--Heath-Brown \cite{BBR} and   Schindler \cite{S2} that we have already described. 
The main idea is to begin by proving an asymptotic formula for the number of (restricted) integer solutions  $(\x,\y)\in \ZZp^n\times \ZZp^n$ to the equation $F(\x;\y)=0$, with $|\x|\leq X$ and $|\y|\leq Y$, for arbitrary $X,Y\geq 1$. When $X$ and $Y$ have a similar size we follow \cite{S2} and invoke the main result in ancillary 
work of Schindler \cite{S1}. On the other hand, when $X$ is much smaller than  $Y$ we shall fix a suitable choice of $\x$ and count zeros of the resulting quadratic form in $\y$. It will be crucial to have to hand an  asymptotic formula for counting rational points of bounded height on quadrics, in which the dependence of the error term on the size of the coefficients is both explicit and as sharp as possible. 
Our primary tool here is the  modern form of the  circle method, due to Duke, Friedlander and Iwaniec \cite{D}, as later refined by 
 Heath-Brown \cite{HB}.
This part of the paper may be of independent interest and is addressed in the next section; it is here  that the saving over Schindler's work for  bidegree $(2,2)$ can be traced back to.
Finally, the case in which $Y$ is much smaller than $X$ follows by symmetry.

\subsection*{Acknowledgements}
The authors are grateful to Kevin Hughes for suggesting the proof of Lemma 
\ref{lem:sig}.
This research was supported by the Natural Science Foundation of China (grant no.\ 11761048) and supported by the China Scholarship Council. While working on this paper 
the  first author was  supported  by EPSRC grant \texttt{EP/P026710/1}.

\section{Quadrics and the circle method}

Let $n\geq 5$ and let $F\in\mathbb{Z}[x_1,\dots,x_n]$ be a non-singular  quadratic form. Thus there is a symmetric matrix $\mathbf{M}=(M_{ij})_{1\leq i,j\leq n}$, with elements in $\mathbb{Z}$, such that
$
F(\mathbf{x})=\mathbf{x}^T\mathbf{M}\mathbf{x}.
$
 We will write $\Delta_F=\det \mathbf{M}$ for its discriminant and
$$
\| F\|=\max_{1\leq i,j\leq n}|M_{ij}|
$$
for its height. 
In this section we want to study the sum
\begin{equation}\label{eq:NN}
N(w;B)=\sum_{\substack{\mathbf{x}\in \mathbb{Z}^n\\ F(\mathbf{x})=0}} w(B^{-1}\mathbf{x}),
\end{equation}
as $B\to \infty$, 
for a suitable class of smooth weight functions $w: \mathbb{R}^n\rightarrow \mathbb{R}$ with compact support. In \S \ref{s:3} we shall see how to deduce an asymptotic formula when 
 $w$ is taken to be the characteristic function of $[-1,1]^n$.

We require  an asymptotic formula in which the error term has an explicit 
dependence on  $\|F\|$.    This type of result is by no means new. 
For example, Browning and Dietmann 
 \cite[Prop. 1]{BD} obtain such an asymptotic formula for 
$N(w_F;B)$ for a very specific weight function $w_F$  that depends on the coefficients of the quadratic form. In our work, on the other hand, we want to work with as general a class of weight functions as possible. 

The leading constant in our result will involve  the 
singular integral
\begin{equation}\label{eq:sigm-inf}
\sigma_{\infty}(w;F)=\int_{-\infty}^{\infty}\int_{\mathbb{R}^{n}} w(\mathbf{x})e(-\theta F(\mathbf{x}))\mathrm {d}\mathbf{x}\mathrm {d}\theta
\end{equation}
and the singular series
$$
\mathfrak{S}(F)=\prod_p \lim_{r\rightarrow \infty}p^{-r(n-1)}\#\{\mathbf{x} \bmod{p^r}: F(\mathbf{x})\equiv 0 \bmod{p^r}\}.
$$
Both of these are convergent when $n\geq 5$.
We may now record the  main result of this section, as
follows.

\begin{thm}\label{t:2}
Let $n\geq 5$, let $\ve>0$ and 
let $\eta\in (0,\frac{1}{4})$. 
Assume that $w$ is a smooth weight function that is compactly support in $\mathbb{R}^n$ such that 
\begin{equation} \label{eq:eta-condition}
w(\mathbf{x})=0 \ \mathrm{for} \ |\mathbf{x}|\leq\eta.
\end{equation}
Then for  any $B\geq 1$ we have
$$
N(w;B)=\sigma_{\infty}(w;F)\mathfrak{S}(F)B^{n-2}+O(\|F\|^\ve B^\ve \mathcal{E}_F(B)),
$$
where
\begin{equation}\label{eq:def-E}
\mathcal{E}_F(B)=
\| F \|^{\frac{3n+\kappa_n}{4}-\frac{1}{2}}B^{\frac{n-2+\kappa_n}{2}}+
 \| F \|^{\frac{3n-\kappa_n}{4}-1}B^{\frac{n-\kappa_n}{2}}.
\end{equation}
and 
\begin{equation} \label{eq:kappa}
\kappa_n=\begin{cases}
           0, & 2\mid n, \\
           1, & 2\nmid n.
         \end{cases}
\end{equation}
\end{thm}

All the implied constants in this section are allowed to depend on the weight function $w$ and on $n$, $\eta$ and $\ve$, but on nothing else. Any other dependence will be indicated by a subscript.  Furthermore, we shall follow common convention and allow the value of $\ve$ to change at different parts of the argument.

It would not be very hard to extend 
Theorem \ref{t:2} to handle quadratic forms in $n= 4$ variables. This is achieved for diagonal $F$ in  work of Browning and Heath-Brown \cite[Thm.~4.1]{BBR}, where an even better error term is available through  exploiting the diagonal structure of the form.

Our proof of Theorem \ref{t:2} uses the smooth $\delta$-function variant of the circle method in the form developed by Heath-Brown \cite{HB}, which we proceed to  summarise here.
For any $q \in \mathbb{N}, \mathbf{c} \in \mathbb{Z}^n$ and $Q>1$, we put 
$$
S_q(\mathbf{c})=\sum_{\substack{a \bmod{ q} \\ 
\gcd(a,q)=1}}
\sum_{\mathbf{b} \bmod{q}} e_q(aF(\mathbf{b})+\mathbf{b}_\cdot \mathbf{c})
$$
and 
$$
I_q(\mathbf{c})=\int_{\mathbb{R}^n}w(B^{-1}\mathbf{x})h\le(\frac{q}{Q},\frac{F(\mathbf{x})}{Q^2}\ri)e_q(-\mathbf{c}_\cdot \mathbf{x})\mathrm{d}\mathbf{x},
$$
for a certain function $h:(0,\infty)\times \mathbb{R}\rightarrow \mathbb{R}$ described in  \cite[$\S 3$]{HB}. All we need to know here is that $h(x,y)$ is infinitely differentiable for $(x,y)\in(0,\infty)\times \mathbb{R}$, and that $h(x,y)$ is non-zero only for $x\leq \min\{1,2|y|\}$.
It now  follows from \cite[Theorem 2]{HB} that
\begin{equation}\label{eq:start}
N(w;B)=\frac{c_Q}{Q^2}\sum_{\mathbf{c}\in \mathbb{Z}^n}\sum_{q=1}^{\infty}q^{-n}S_q(\mathbf{c})I_q(\mathbf{c}),
\end{equation}
where the constant $c_Q$ satisfies $c_Q=1+O_N(Q^{-N})$, for any $N>0$. We shall take $Q=\sqrt{\| F\| B^2}$ in our work.

\subsection{The integral $I_q(\mathbf{c})$}

We want to use the bounds for the exponential integral $I_q(\mathbf{c})$ from  \cite[\S\S 7--8]{HB}. 
However, we shall require versions in which the dependence on the coefficients of $F $ is made explicit, as follows.

\begin{lem}\label{lem:integral}
Let $\mathbf{c} \in \mathbb{Z}^n$ be non-zero. Then
\begin{enumerate}
\item
for any $N>0$, we have
$$
I_q(\mathbf{c})\ll_N \frac{B^{n+1}}{q}\frac{\| F\|^{N/2+1/2}}{|\mathbf{c}|^{N}};
$$
\item
we have
$$
I_q(\mathbf{c})\ll \| F\|^{\frac{n}{2}}|\Delta_F|^{-\frac{1}{2}}B^{\frac{n}{2}+1}
q^{\frac{n}{2}-1}|\mathbf{c}|^{-\frac{n}{2}+1}.
$$
\end{enumerate}
\end{lem}
\begin{proof}
This is a straightforward adaptation to general $n$ from the case $n=4$ that is treated in \cite[Lemma 4.2]{BBR}.
\end{proof}

Our next task is to relate $I_{q}(\mathbf{0})$ to the singular integral, given by 
\eqref{eq:sigm-inf}. We shall closely follow the argument given in \cite[\S 4.3]{BBR}.
To begin with we note that
$$
I_q(\mathbf{0})=B^n\int_{\mathbb{R}^n}w(\mathbf{x})h(q/Q,G(\mathbf{x}))\mathrm{d}\mathbf{x},
$$
where $G(\mathbf{x})=\| F\|^{-1}F(\mathbf{x})$.
Arguing as in
\cite[\S 4.3]{BBR},
 we choose a smooth weight function $w_1: \mathbb{R}\rightarrow \mathbb{R}$ 
 such that 
$$
I_q(\mathbf{0})=B^n\int_{\mathbb{R}^n}w(\mathbf{x})w_1(G(\mathbf{x}))h(q/Q,G(\mathbf{x}))\mathrm{d}\mathbf{x},
$$
with  $f(t)=w_1(t)h(q/Q,t)$ being compactly supported with a continuous second derivative.
Let
\begin{equation}
\label{eq:KKK}
K(u;\delta)=\begin{cases}
                \delta^{-2}(\delta-|u|), &  |u|\leq \delta,\\
                0, &  |u|\geq \delta.
\end{cases}
\end{equation}
We can write  $K(u;\delta)$ in terms of its Fourier transform as
$$
K(u;\delta)=\int_{-\infty}^{\infty}e(\theta u)\le(\frac{\sin (\pi \delta \theta)}{\pi \delta \theta}\ri)^2 \mathrm{d}\theta.
$$
Then
\begin{equation}\label{eq:limit}
\lim_{\delta\downarrow 0}
\int_{-\infty}^{\infty}f(t)K(t-\tau; \delta)\mathrm{d}t=
f(\tau),
\end{equation}
uniformly in $\tau$. Thus 
\begin{equation}\label{eq:Iq0}
I_q(\mathbf{0})
=B^n\lim_{\delta\downarrow 0}\int_{-\infty}^{\infty}\le(\frac{\sin (\pi \delta \theta)}{\pi \delta \theta}\ri)^2J(\theta;w)L(\theta)\mathrm{d}\theta,
\end{equation}
with
$$
J(\theta;w)=\int_{\mathbb{R}^n}w(\mathbf{x})e(-\theta G(\mathbf{x}))\mathrm{d}\mathbf{x}
\quad \text{and}\quad
L(\theta)=\int_{-\infty}^{\infty}w_1(t)h(q/Q,t)e(\theta t)\mathrm{d}t.
$$

It follows from \cite[Lemma 4.11]{BBR} that 
\begin{equation}\label{eq:L411}
L(\theta)=1+O_N((q/Q)^N)+O_N((q/Q)^N|\theta|^N),
\end{equation}
for any $N\in \NN$, provided that 
 $q\ll Q$.
Next, 
if $w$ is smooth and supported in $[-\kappa,\kappa]^n$, then 
\cite[Lemma 3.1]{HBP} implies that 
$$
J(\theta;w)\ll_{\kappa}\prod_{i=1}^n \min\{1, |\theta \|F\|^{-1} \rho_i|^{-\frac{1}{2}}\},
$$
where $\rho_1,\dots, \rho_n$ are the eigenvalues of the matrix $\mathbf{M}$ that defines  $F$. Since $|\rho_1 \dots \rho_n|=|\Delta_F|$, it readily follows that 
\begin{equation}\label{eq:J-estimate}
J(\theta;w)\ll \min\{1, |\theta|^{-\frac{n}{2}}\|F\|^{\frac{n}{2}}|\Delta_F|^{-\frac{1}{2}}\}.
\end{equation}

We may now establish the following result.

\begin{lem}
\label{lem:i0}
 Let $N\in \ZZ$ such that  $0< N\leq n/2-1$. Then
$$
I_q(\mathbf{0})=B^n\left\{\sigma_{\infty}(w;F)\| F\|+O\left(\| F\|^{N+1+\ve}|\Delta_F|^{-\frac{N+1}{n}}\left(\frac{q}{Q}\right)^{N}\right)\right\}.
$$
\end{lem}

\begin{proof}
We return to \eqref{eq:Iq0} and invoke 
\eqref{eq:L411} and 
\eqref{eq:J-estimate}.
The error terms in the former contribute
\begin{align*}
&\ll B^n\left(\frac{q}{Q}\right)^{N}\int_{-\infty}^{\infty}(1+|\theta|^{N})\min\bigg {\{}1, \frac{\|F\|^{\frac{n}{2}}}{|\theta|^{\frac{n}{2}}|\Delta_F|^{\frac{1}{2}}}\bigg {\}}\mathrm{d}\theta\\
&\ll B^n\left(\frac{q}{Q}\right)^{N}\|F\|^{N+1+\ve}|\Delta_F|^{-\frac{N+1}{n}},
\end{align*}
for any $N\leq n/2-1$. 
The lemma easily follows since
$$
\int_{-\infty}^{\infty}\le(\frac{\sin (\pi \delta \theta)}{\pi \delta \theta}\ri)^2J(\theta;w)\mathrm{d}\theta\rightarrow \int_{-\infty}^{\infty}J(\theta;w)\mathrm{d}\theta,
$$
as $\delta\downarrow 0$.  
\end{proof}

We will also need a good upper bound for the singular integral $\sigma_\infty(w;F)$ and this is provided by the following result. 

\begin{lem}\label{lem:sig}
Suppose either that $w$ is a smooth weight function of compact support, or that $w$ is the characteristic function
of $[-\kappa,\kappa]^n$ for some $\kappa>0$. Then
$
 \sigma_{\infty}(w;F)\ll_{\kappa} |\Delta_F|^{-\frac{1}{n}+\ve}
$, for any $\ve>0$.
\end{lem}

\begin{proof}
We need to study 
$$
J(\theta;w)=\int_{\mathbb{R}^n}w(\mathbf{x})e(-\theta G(\mathbf{x}))\mathrm{d}\mathbf{x},
$$
for the weights $w$ in the statement of the lemma.
If $w$ is smooth and supported in $[-\kappa,\kappa]^n$, then we have \eqref{eq:J-estimate}, where the implied constant is allowed to depend on $\kappa$.
In this case the  statement  follows on integrating over $\theta$.

Next,  if $w$ is the characteristic function of $[-\kappa,\kappa]^n$, then 
$$
J(\theta;w)=\int_{[-\kappa,\kappa]^n}e(-\theta G(\mathbf{x}))\mathrm{d}\mathbf{x}.
$$
We adopt a Weyl differencing approach to estimating this integral. Thus
\begin{align*}
|J(\theta;w)|^2
&\leq 
\int_{|\z|\leq 2\kappa} \left|
\int_{|\x|\leq \kappa, ~|\x+\z|\leq \kappa}
e(-\theta \x.\nabla G(\z))\mathrm d\x\right|\mathrm d\z\\
&\ll_\kappa
\int_{|\z|\leq 2\kappa} \prod_{i=1}^n
\min\left\{1, \left|\theta  \frac{\partial G(\z)}{\partial x_i}\right|^{-1}\right\}.
\end{align*}
Now 
$\frac{\partial G(\z)}{\partial x_i}\ll_\kappa 1$ for all $1\leq i\leq n$ and all $|\z|\leq 2\kappa$.
Let $\mathbf{r}=(r_1,\dots,r_n)\in \ZZ^n$
 such that 
$0\leq r_1,\dots,r_n\ll_\kappa |\theta|$  and let 
$$
S(\mathbf{r})=\left\{ |\z|\leq 2\kappa: 
\frac{r_i}{|\theta|}<
\left|  \frac{\partial G(\z)}{\partial x_i}\right| \leq \frac{r_i+1}{|\theta|}
\text{ for $1\leq i\leq n$}
\right\}.
 $$ 
Then 
\begin{align*}
|J(\theta;w)|^2
&\ll_\kappa \sum_{0\leq r_1,\dots,r_n\ll_\kappa |\theta|}
\meas\left( S(\mathbf{r})\right)
 \prod_{i=1}^n
\min\left\{1, \frac{1}{r_i}\right\}\\
&\ll_\kappa \log^n(2+|\theta|)
\meas\left( S(\mathbf{r})\right).
\end{align*}

We claim that 
$$
\meas\left( S(\mathbf{r})\right)\ll_\kappa \min\left\{1 , \frac{\|F\|^n}{|\Delta_F||\theta|^n}\right\}.
$$
This will suffice for the lemma, on integrating over $\theta$, since our work above then shows that 
$$
J(\theta;w)
\ll_\kappa 
 \min\{1, |\theta|^{-\frac{n}{2}}\|F\|^{\frac{n}{2}}|\Delta_F|^{-\frac{1}{2}}\}
\log^{\frac{n}{2}}(2+|\theta|).
$$
To check the claim, we note that the bound $\meas\left( S(\mathbf{r})\right)=O_\kappa(1)$ is trivial. To provide a second estimate, we note that if $S(\mathbf{r})=\emptyset$ then there is nothing to prove. On the other hand, 
if $\mathbf{t}\in S(\mathbf{r})$ then we can make the change of variables $\z\to \z-\mathbf{t}$.
Since the partial derivatives are linear forms it readily follows that 
\begin{align*}
\meas \left(S(\mathbf{r})\right)
&\ll
\meas \left\{ |\z|\leq 4\kappa: 
\left|  \nabla G(\z)\right| \leq \frac{1}{|\theta|}
\right\} \\ &\leq \frac{1}{|\Delta_G|} 
\meas \left\{ |\u|\leq  \frac{1}{|\theta|}
\right\} \\
&\ll \frac{1}{|\Delta_G||\theta|^n} .
\end{align*}
The claim follows on noting that $\Delta_G=\|F\|^{-n}\Delta_F$.
\end{proof}

\subsection{The exponential sums}

In this section we summarise what we  need to know about the exponential sum $S_q(\mathbf{c})$. First, it follows from  \cite[Lemma~23]{HB} that 
$S_{q_1q_2}(\mathbf{c})=S_{q_1}(\mathbf{c})S_{q_2}(\mathbf{c})$ for any coprime integers $q_1, q_2$.
The following upper bound is standard (and 
is the  case $k=0$ of \cite[Lemma 4]{BD}, for example).

\begin{lem}\label{lem:upper}
Let $\ve>0$. Then 
$
S_q(\mathbf{c})\ll q^{\frac{n}{2}+1+\ve}\gcd(q^n,\Delta_F)^{\frac{1}{2}}.
$
\end{lem}

The dual quadratic form $F^*$ is the quadratic form whose matrix is $\Delta_F\mathbf{M}^{-1}$. When $p\nmid 2\Delta_F$ we may think of $F^*$ as being defined modulo $p^r$ for any $r\in \NN$. 
The following explicit formula follows from applying modulus signs in  work of  Browning and Munshi \cite[Lemma 15]{BM}.

\begin{lem}\label{lem:pr}
Let $p\nmid 2\Delta_F$ be a prime and let $r\in \NN$. Then 
$$
|S_{p^r}(\mathbf{c})|=\begin{cases}
                        p^{\frac{nr}{2}}|c_{p^r}(-\overline{4}F^{*}(\mathbf{c}))|, & 2\mid nr, \\
                        p^{\frac{nr}{2}}|g_{p^r}(-\overline{4}F^{*}(\mathbf{c}))|, & 2\nmid nr,
                      \end{cases}
$$
where $c_{p^r}(\cdot)$ is  the Ramanujan sum and 
$$
g_{p^r}(\cdot)=\sum_{a \bmod{p^r}}
\left(\frac{a}{p^r}\right)
e_{p^r}(a\cdot ).
$$
\end{lem}

We are now ready to study the asymptotic behaviour of the sum
$$
\Sigma_n(x;\mathbf{c})=\sum_{x/2<q\leq x}|S_{q}(\mathbf{c})|,
$$
for suitable vectors $\mathbf{c}\in \ZZ^n$.

\begin{lem}\label{lem:av1}
Let $\ve>0$ and assume that $F^{*}(\mathbf{c})\neq 0$. Then
$$\Sigma_n(x;\mathbf{c})\ll 
|\Delta_F|^{\frac{1}{2}+\ve}
x^{\frac{n+\kappa_n}{2}+1+\ve}
|\mathbf{c}|^\ve,
$$
where $\kappa_n$ is given by \eqref{eq:kappa}.
\end{lem}
\begin{proof}
Define the non-zero integer $N=2\Delta_F F^{*}(\mathbf{c})$. Then
\begin{equation}
\label{eq:begin}
\Sigma_n(x;\mathbf{c})\leq \sum_{\substack{q_2\leq x \\ q_2|N^\infty}}
|S_{q_2}(\mathbf{c})|\sum_{\substack{q_1\leq x/q_2 \\
\gcd(q_1, N)=1}}
|S_{q_1}(\mathbf{c})|.
\end{equation}
Since $\gcd(q_1,N)=1$,  we  have 
$|c_{q_1}(-\overline{4}F^{*}(\mathbf{c}))|=1$ and 
$|g_{q_1}(-\overline{4}F^{*}(\mathbf{c}))|=\sqrt{q_1}$.
Hence Lemma \ref{lem:pr} yields $S_{q_1}(\mathbf{c})\ll q_1^{\frac{n+\kappa_n}{2}}$ and it follows from  Lemma \ref{lem:upper} that
\begin{align*}
\Sigma_n(x;\mathbf{c})&\ll x^{\frac{n+\kappa_n}{2}+1}
\sum_{\substack{q_2\leq x \\ q_2|N^\infty}}
q_2^{-\frac{n+\kappa_n}{2}-1}
|S_{q_2}(\mathbf{c})|\\
&\ll x^{\frac{n+\kappa_n}{2}+1} \sum_{\substack{q_2\leq x \\ q_2|N^\infty}}q_2^{-\frac{\kappa_n}{2}+\ve} \gcd(q_2^n,\Delta_F)^{1/2}.
\end{align*}
The statement of the lemma 
follows on noting that 
$\gcd(q_2^n,\Delta_F)^{1/2}\leq  |\Delta_F|^{\frac{1}{2}}$
and  there are $O(N^{\ve}x^{\ve})$ values of $q_2$ that contribute to the remaining sum.
\end{proof}

\begin{lem}\label{lem:av2}
Let $\ve>0$ and assume that $F^{*}(\mathbf{c})=0$. Then
$$
\Sigma_n(x;\mathbf{c})\ll|\Delta_{F}|^{\frac{1}{2}-\frac{1}{n}+\frac{\kappa_n}{2n}+\ve}
x^{\frac{n}{2}+2-\frac{\kappa_n}{2}+\ve}.
$$
\end{lem}
\begin{proof}
We begin as in \eqref{eq:begin}, but now with the  non-zero integer $N=2\Delta_{F}$.
If $2\mid n$ then it  follows from Lemma \ref{lem:pr} that
$$
\sum_{\substack{q_1\leq x/q_2\\ \gcd(q_1,N)=1}}|S_{q_1}(\mathbf{c})| \leq \sum_{q_1\leq x/q_2} 
q_1^{\frac{n}{2}+1} \ll \left(\frac{x}{q_2}\right)^{\frac{n}{2}+2}.
$$
On the other hand, if $2\nmid n$ then only square values of $q_1$ contribute to the sum, so that
$$
\sum_{\substack{q_1\leq x/q_2\\ \gcd(q_1,N)=1}}|S_{q_1}(\mathbf{c})| \leq \sum_{\substack{
q_1\leq x/q_2\\
q_1 =\square}} 
q_1^{\frac{n}{2}+1} \ll \left(\frac{x}{q_2}\right)^{\frac{n}{2}+\frac{3}{2}}.
$$
Thus Lemma \ref{lem:upper} yields
\begin{align*}
\Sigma_n(x;\mathbf{c})\ll   x^{\frac{n}{2}+2-\frac{\kappa_n}{2}} \sum_{\substack{q_2\leq x\\ q_2\mid N^\infty}}\frac{\gcd(q_2^n,\Delta_{F})^{\frac{1}{2}}}{q_2^{1-\frac{\kappa_n}{2}-\ve}}
&\ll  x^{\frac{n}{2}+2-\frac{\kappa_n}{2}}\sum_{\substack{q_2\leq x \\ q_2|N^\infty}}q_2^\ve|\Delta_{F}|^{\frac{1}{2}-\frac{1}{n}+\frac{\kappa_n}{2n}}.
\end{align*}
The statement of the lemma follows since there are $O(N^{\ve}x^{\ve})$ choices for $q_2$.
\end{proof}

\subsection{Proof of Theorem 2.1}

Let $\ve>0$ and let $C=\|F\|^{\frac{1}{2}}Q^{\ve}$. Returning to \eqref{eq:start}, 
part (1) of Lemma \ref{lem:integral} gives
$$
N(w;B)=M(B)+E(B)+O(1),
$$
where
$$
M(B)=\frac{1}{Q^2}\sum_{q=1}^{\infty}q^{-n}S_q(\mathbf{0})I_q(\mathbf{0}), 
\quad
E(B)=\frac{1}{Q^2}\sum_{\substack{\mathbf{c}\in \mathbb{Z}^n \\ 0<|\mathbf{c}|\ll C}}\sum_{q=1}^{\infty}q^{-n}S_q(\mathbf{c})I_q(\mathbf{c}).
$$
Note that  $I_q(\mathbf{c})$ vanishes for $q\gg Q$, by the properties of the $h$-function.

Taking $N=n/2-2+\kappa_n/2$ in Lemma \ref{lem:i0}, 
 we obtain
\begin{align*}
M(B)
=~&\frac{\|F\|\sigma_{\infty}(w;F)B^n}{Q^2}\sum_{q=1}^{\infty}q^{-n}S_q(\mathbf{0})+O(T)\\
\quad &+O\le(\frac{\|F\|^{\frac{n}{2}-1+\frac{\kappa_n}{2}+\ve}|\Delta_F|^{-\frac{n-2+\kappa_n}{2n}}B^{n}}{Q^{\frac{n}{2}+1+\frac{\kappa_n}{2}}}\sum_{q\ll Q}
q^{-\frac{n}{2}-1+\frac{\kappa_n}{2}}|S_q(\mathbf{0})|\ri),
\end{align*}
where
$$
T=\frac{\|F\|\sigma_{\infty}(w;F)B^n}{Q^2}\sum_{q\gg Q}q^{-n}|S_q(\mathbf{0})|.
$$
Breaking into dyadic intervals, it  follows from Lemma \ref{lem:av2} that 
\begin{align*}
\sum_{q\gg Q}q^{-n}|S_q(\mathbf{0})| &\ll 
\sum_{\substack{x\gg Q\\ \text{$x$ dyadic}}} x^{-n}~\Sigma_n(x;\0)
\ll 
|\Delta_F|^{\frac{n-2+\kappa_n}{2n}+\ve} Q^{2-\frac{n}{2}-\frac{\kappa_n}{2}+\ve},
\end{align*}
Applying Lemma~\ref{lem:sig}, we therefore obtain
\begin{align*}
T &\ll \frac{\|F\||\Delta_F|^{-\frac{1}{n}+\ve}B^n}{Q^2}
\cdot
|\Delta_F|^{\frac{n-2+\kappa_n}{2n}+\ve} Q^{2-\frac{n}{2}-\frac{\kappa_n}{2}+\ve}\\
&\ll \|F\|^{-\frac{n}{4}+1+\ve}|\Delta_{F}|^{\frac{n-4+\kappa_n}{2n}}B^{\frac{n-\kappa_n}{2}+\ve}\\
&\ll
\|F\|^{\frac{n}{4}-1+\frac{\kappa_n}{2}+\ve}B^{\frac{n-\kappa_n}{2}+\ve},
\end{align*}
since  $|\Delta_F|\ll \| F \|^n$.

Similarly, 
\begin{align*}
\sum_{q\ll Q}
q^{-\frac{n}{2}-1+\frac{\kappa_n}{2}}|S_q(\mathbf{0})|
&\ll 
\sum_{\substack{x\ll Q\\ \text{$x$ dyadic}}} x^{-\frac{n}{2}-1+\frac{\kappa_n}{2}}~\Sigma_n(x;\0)\ll  |\Delta_{F}|^{\frac{n-2+\kappa_n}{2n}+\ve}Q^{1+\ve}.
\end{align*}
The remaining   error term in our estimate  $M(B)$ is therefore seen to be 
\begin{align*}
&\ll
\frac{\|F\|^{\frac{n}{2}-1+\frac{\kappa_n}{2}+\ve}|\Delta_F|^{-\frac{n-2+\kappa_n}{2n}}B^{n}}{Q^{\frac{n}{2}+1+\frac{\kappa_n}{2}}}
\cdot
 |\Delta_{F}|^{\frac{n-2+\kappa_n}{2n}+\ve}Q^{1+\ve}
\ll \|F\|^{\frac{n}{4}-1+\frac{\kappa_n}{4}+\ve}B^{\frac{n-\kappa_n}{2}+\ve}.
\end{align*}
It follows that 
$$
M(B)=\sigma_{\infty}(w;F)\mathfrak{S}(F)B^{n-2}+O\le(
\|F\|^{\frac{n}{4}-1+\frac{\kappa_n}{2}+\ve}B^{\frac{n-\kappa_n}{2}+\ve}\ri),
$$
where 
$\mathfrak{S}(F)=\sum_{q=1}^{\infty}q^{-n}S_q(\mathbf{0})$ is the usual singular 
series.

Now we deal with $E(B)$.  Putting $C=\|F\|^{\frac{1}{2}}Q^{\ve}$, we deduce from part (2) of Lemma \ref{lem:integral} 
that 
\begin{equation}
\begin{split}
\label{eq:EE}
E(B)
&\ll \frac{\| F \|^{\frac{n}{2}}}{Q^2 |\Delta_F|^{\frac{1}{2}}}B^{\frac{n}{2}+1}\sum_{\substack{\mathbf{c}\in \mathbb{Z}^n \\ 0<|\mathbf{c}|\ll C }}\sum_{q\ll Q}q^{-\frac{n}{2}-1}|S_q(\mathbf{c})||\mathbf{c}|^{-\frac{n}{2}+1}\\
&\ll \frac{\| F \|^{\frac{n}{2}}}{Q^2 |\Delta_F|^{\frac{1}{2}}}B^{\frac{n}{2}+1}\sum_{\substack{\mathbf{c}\in \mathbb{Z}^n \\ 
0<
|\mathbf{c}|\ll C }}
|\mathbf{c}|^{-\frac{n}{2}+1}
\sum_{\substack{x\ll Q\\ \text{$x$ dyadic}}} x^{-\frac{n}{2}-1}
\Sigma_n(x;\mathbf{c}).
\end{split}
\end{equation}
Let $E_1(B)$ denote the  contribution from vectors $\c$ for which  $F^{*}(\mathbf{c})\neq 0$ and let $E_2(B)$ be the contribution from non-zero vectors $\c$ for which $F^*(\c)=0$.
It follows from Lemma \ref{lem:av1} that
\begin{align*}
E_1(B)
&\ll \frac{\| F \|^{\frac{n}{2}}}{Q^2|\Delta_F|^{\frac{1}{2}}}B^{\frac{n}{2}+1}Q^{\frac{\kappa_n}{2}+\ve}|\Delta_F|^{\frac{1}{2}+\ve}\sum_{\substack{\mathbf{c}\in \mathbb{Z}^n \\ |\mathbf{c}|\ll C \\ F^{*}(\mathbf{c})\neq 0}}|\mathbf{c}|^{-\frac{n}{2}+1+\ve}\\
&\ll \| F \|^{\frac{3n}{4}-\frac{1}{2}+\frac{\kappa_n}{4}+\ve}B^{\frac{n-2+\kappa_n}{2}+\ve}.
\end{align*}
Next, starting from \eqref{eq:EE},  it follows from
Lemma~\ref{lem:av2} that 
\begin{align*}
E_2(B)
&\ll \frac{\| F \|^{\frac{n}{2}}}{Q^2 |\Delta_F|^{\frac{1}{2}}}B^{\frac{n}{2}+1}
|\Delta_{F}|^{\frac{n-2+\kappa_n}{2n}+\ve}Q^{1-\frac{\kappa_n}{2}+\ve}
\sum_{\substack{\mathbf{c}\in \mathbb{Z}^n \\ 0<|\mathbf{c}|\ll C \\ 
F^{*}(\mathbf{c})= 0}}|\mathbf{c}|^{-\frac{n}{2}+1+\ve}.
\end{align*}
According to a result of  Heath-Brown \cite[Thm.~2]{HB1},  there are $O(C^{n-2+\ve})$  vectors $\mathbf{c}\in \mathbb{Z}^n$ for which $|\mathbf{c}|\ll C$ and $F^{*}(\mathbf{c})=0$. It is important to emphasise here that the implied constant depends only on $\ve$ and $n$.
Hence 
\begin{align*}
E_2(B)
&\ll \frac{\| F \|^{\frac{n}{2}}}{Q^2 |\Delta_F|^{\frac{1}{2}}}B^{\frac{n}{2}+1}
|\Delta_{F}|^{\frac{n-2+\kappa_n}{2n}+\ve}Q^{1-\frac{\kappa_n}{2}+\ve}
C^{\frac{n}{2}-1+\ve}\\
&\ll \frac{\| F \|^{\frac{3n}{4}-\frac{1}{2}+\ve}B^{\frac{n}{2}+1+\ve}}{Q^{1+\frac{\kappa_n}{2}}}\\
&\ll \| F \|^{\frac{3n-\kappa_n}{4}-1+\ve}B^{\frac{n-\kappa_n}{2}+\ve}.
\end{align*}
This completes the proof of Theorem \ref{t:2}.

\subsection{Removing the smooth weights}\label{s:3}

The counting function $N(w;B)$  in \eqref{eq:NN} was 
studied  for a smooth  weight function $w$. In this section we show how it can be removed, following a line of attack found in \cite[\S 5.3]{BBR}.

We begin by analysing the 
 singular integral $\sigma_\infty(w;F)$ that appears in the  \eqref{eq:sigm-inf}.
Suppose that  $w$ is a smooth weight function of compact support or the characteristic function of 
$[-1,1]^n$. 
We note that
\begin{align*}
\|F\|\sigma_\infty(w;F)&=
\int_{-\infty}^{\infty}J(\theta;w)\mathrm{d}\theta \\&= \lim_{\delta\downarrow 0}\int_{-\infty}^{\infty}\le(\frac{\sin (\pi \delta \theta)}{\pi \delta \theta}\ri)^2 J(\theta;w)\mathrm{d}\theta\\
&= \lim_{\delta\downarrow 0}\int_{-\infty}^{\infty}\int_{\mathbb{R}^n}w(\mathbf{x})e(-\theta G(\mathbf{x}))\le(\frac{\sin (\pi \delta \theta)}{\pi \delta \theta}\ri)^2 \mathrm{d}\mathbf{x}\mathrm{d}\theta,
\end{align*}
in the notation that was introduced after Lemma \ref{lem:integral}.
Arguing as in the proof of Lemma \ref{lem:sig}, 
the conditions for Fubini's Theorem are satisfied, which therefore allows us to switch the two integrations. It therefore follows from \eqref{eq:limit} that
\begin{equation}\label{eq:3.1}
\sigma_\infty(w;F)=
\int_{-\infty}^{\infty}J(\theta;w)\mathrm{d}\theta = \lim_{\delta\downarrow 0}\int_{\mathbb{R}^n}w(\mathbf{x})K(-F(\mathbf{x});\delta)\mathrm{d}\mathbf{x}.
\end{equation}
In particular 
$\sigma_\infty(w;F)\geq 0$.

We are now ready to prove the following result.

\begin{lem}\label{lem:remove}
Let $n\geq 5$, let $\eta\in (0,\frac{1}{4})$ and let $\ve>0$. Let $w_0(\mathbf{x})$ be the characteristic function of the region $[-1,1]^n$. Suppose that $w_1(\mathbf{x})$ (respectively, $w_2(\mathbf{x})$) is a smooth weight function  that is supported in the region $\eta\leq |\mathbf{x}|\leq 1$ (respectively, $\eta\leq |\mathbf{x}|\leq 1+\eta$) and takes values in $[0,1]$ there. Suppose further that $w_1(\mathbf{x})=1$ whenever $2\eta\leq |\mathbf{x}|\leq 1-\eta$ (respectively, $w_2(\mathbf{x})=1$ whenever $2\eta\leq |\mathbf{x}|\leq 1$). Then 
\begin{align*}
N(w_0;B)=~&\big(1+O(\eta)\big)\sigma_{\infty}(w_0;F)\mathfrak{S}(F)B^{n-2}
+O(\|F\|^\ve B^\ve\mathcal{E}_F(B)),
\end{align*}
where  $\mathcal{E}_F( B)$ is given by \eqref{eq:def-E}.
\end{lem}

\begin{proof}
Let
 $\eta<\frac{1}{4}$. 
Let $w_1$, $w_2$ be weights depending on $\eta$ alone, which  satisfy 
\eqref{eq:eta-condition} 
and the conditions of the  lemma.
Thus  $0\leq w_1(\mathbf{u}),w_2(\mathbf{u})\leq 1$
for all $\mathbf{u}\in \RR^n$ and both functions vanish when $|\mathbf{u}|\leq \eta$. The weight $w_1$ takes the value $1$ for $2\eta\leq |\mathbf{u}|\leq 1-\eta$ and vanishes for $|\mathbf{u}|\geq 1$; the weight $w_2$ takes the value $1$ for $2\eta\leq |\mathbf{u}|\leq 1$ and vanishes for $|\mathbf{u}|\geq 1+\eta$. In particular it is now clear that 
$$
0\leq w_1(\mathbf{u})\leq w_0(\mathbf{u})
\leq w_2(\mathbf{u}) +w_0(\mathbf{u}/(2\eta))
$$
for all $\mathbf{u}\in \RR^n$, whence
$$
N(w_1;B)\leq N(w_0;B)\leq N(w_2; B)+ N(w_0;2\eta B).
$$
 Theorem \ref{t:2} implies that 
\begin{align*}
N(w_0;2\eta B)&\ll 1+\sum_{j=0}^{\infty}N(w_2;2\eta B/2^j)\\
&\ll\sigma_{\infty}(w_2;F)\mathfrak{S}(F)(\eta B)^{n-2}+\|F\|^\ve B^\ve\mathcal{E}_F(\eta B)
\end{align*}
and 
$$
N(w_i;B)=\sigma_{\infty}(w_i;F)\mathfrak{S}(F)B^{n-2}+O(\|F\|^\ve B^\ve \mathcal{E}_F(B)),
$$
for $i=1,2$. 
Replacing $\sigma_{\infty}(w_i;F)$ by $\sigma_{\infty}(w_0;F)$ and bringing everything together, we 
conclude that 
\begin{align*}
N(w_0;B)=~&\sigma_{\infty}(w_0;F)\mathfrak{S}(F)B^{n-2}
+O(\sigma_1'\mathfrak{S}(F)B^{n-2})+O(\sigma_2'\mathfrak{S}(F)B^{n-2})\\
&\quad+O(\eta^{n-2} \sigma_{\infty}(w_2;F)\mathfrak{S}(F)B^{n-2})
+O(\|F\|^\ve B^\ve\mathcal{E}_F(B)),
\end{align*}
where
$\sigma'_i=|\sigma_{\infty}(w_i;F)-\sigma_{\infty}(w_0;F)|$ for $i=1,2$.

According to \eqref{eq:3.1} 
we have $\sigma_\infty(w_i;F)=\lim_{\delta \downarrow 0 } \sigma^{(\delta)}(w_i;F)$
for $i=0,1,2$, where
$$
\sigma^{(\delta)}(w_i;F)=
\int_{\mathbb{R}^n}w_i(\mathbf{x})K(-F(\mathbf{x});\delta)\mathrm{d}\mathbf{x}.
$$
Note that 
\begin{align*}
|\sigma^{(\delta)}(w_1;F)-\sigma^{(\delta)}(w_0;F)|
&\leq 
\int_{|\x|\leq 2\eta}
\hspace{-0.4cm}
K(-F(\mathbf{x});\delta)\mathrm{d}\mathbf{x}+
\int_{1-\eta\leq |\x|\leq 1}
\hspace{-0.4cm}
K(-F(\mathbf{x});\delta)\mathrm{d}\mathbf{x}
\end{align*}
and 
$$|\sigma^{(\delta)}(w_2;F)-\sigma^{(\delta)}(w_0;F)|
\leq \int_{|\x|\leq 2\eta}
\hspace{-0.4cm}
K(-F(\mathbf{x});\delta)\mathrm{d}\mathbf{x}
+
\int_{1\leq |\x|\leq 1+\eta}\hspace{-0.4cm}
K(-F(\mathbf{x});\delta)\mathrm{d}\mathbf{x}.
$$
Appealing to \eqref{eq:KKK}, it is clear that  $K(t u;\delta)=t^{-1}K(u;\delta/t)$
for any $t>0$. Hence
$$
\int_{|\x|\leq \lambda}K(-F(\mathbf{x});\delta)\mathrm{d}\mathbf{x}
=\lambda^{n-2}
\sigma^{(\delta/\lambda^2)}(w_0;F),
$$
for any $\lambda>0$.  Thus 
\begin{align*}
|\sigma^{(\delta)}(w_1;F)-\sigma^{(\delta)}(w_0;F)|
\leq ~&
(2\eta)^{n-2}
\sigma^{(\delta/4\eta^2)}(w_0;F)\\
&+\sigma^{(\delta)}(w_0;F)-(1-\eta)^{n-2}\sigma^{(\delta/(1-\eta)^2)}(w_0;F)
\end{align*}
and 
\begin{align*}
|\sigma^{(\delta)}(w_2;F)-\sigma^{(\delta)}(w_0;F)|\leq ~&
(2\eta)^{n-2}
\sigma^{(\delta/4\eta^2)}(w_0;F)\\
&+(1+\eta)^{n-2}\sigma^{(\delta/(1+\eta)^2)}(w_0;F)
-\sigma^{(\delta)}(w_0;F).
\end{align*}
Taking the limit $\delta \downarrow 0$ it is now clear that 
$\sigma_i'=O(\eta \sigma_{\infty}(w_0;F))$ for $i=1,2$, which  thereby leads to the statement of 
the lemma.
\end{proof}

\section{Back to biquadratic hypersurfaces }\label{s:4}

We now begin the proof of  Theorem \ref{t:1} in earnest.  
Recall our definition \eqref{eq:Z} of the closed set $Z$, where
$F_\x$ is the quadratic form in $\y$ obtained 
from $F(\x;\y)$ by fixing $\x$, and  $F_\y$ is obtained from $F(\x;\y)$ by instead fixing $\y$. 
We claim that $Z$ is a proper closed subset of $X$ when $X$ is smooth. 

To see this, we
suppose that $\det(F_\x)$ is identically zero and 
 write  $F(\mathbf{x};\mathbf{y})=\mathbf{y}^T \mathbf{M}(\mathbf{x})\mathbf{y}$ 
 for a suitable underlying matrix $\mathbf{M}(\x)$.
 We may view this as a quadratic form 
 over the field $K=\mathbb{Q}(x_1,\dots,x_n)$.  Since it is  diagonalisable over $K$, we may assume without loss of generality that
$$
F(\mathbf{x};\mathbf{y})=Q_1y_1^2+\dots+Q_ny_n^2,
$$
for rational functions $Q_1,\dots,Q_n\in K$, with $Q_n=0$.
For any $\x^*\neq \0$ we choose $\mathbf{y}^*=(0,\dots,1)$. Then 
all of the partial derivatives of $F(\x;\y)$ vanish at $(\x^*;\y^*)$, which contradicts the fact that $X$ is smooth. Thus $\det(F_\x)$ is not identically zero, so that $Z$ is indeed a proper closed subvariety of $X$.

Let $X,Y>1$, with $X\leq Y$.
Our starting point is going to be an  asymptotic formula for the number of integer solutions to $F(\mathbf{x};\mathbf{y})=0$ with $|\mathbf{x}|\leq X$ and $|\mathbf{y}|\leq Y$, restricted to the sets
$$
\mathcal{A}_1=\{\mathbf{x}\in \mathbb{Z}^n: \det(F_{\mathbf{x}})\neq 0\}, \quad \mathcal{A}_2=\{\mathbf{y}\in \mathbb{Z}^n: \det(F_{\mathbf{y}})\neq 0\}.
$$
We  put $ \mathcal{A}=\mathcal{A}_1\times \mathcal{A}_2$ and 
define  $N(\mathcal{A};X,Y)$ to be the number of vectors $\mathbf{x}\in [-X,X]\cap \mathcal{A}_1$ and $\mathbf{y}\in [-Y,Y]\cap \mathcal{A}_2$ such that $F(\mathbf{x};\mathbf{y})=0$. 

We will adopt two approaches according to the relative size of $X$ and $Y$. First, if $X$ and $Y$ are roughly of the same size we apply  Schindler's work \cite{S1}. Second, if $X$ is substantially smaller than $Y$, then for each  fixed $\mathbf{x}$ we count the number of integer solutions $\mathbf{y}$ 
using  Lemma \ref{lem:remove}, before summing over the relevant $\x$.
It will be convenient to introduce the notation
$$
u=\frac{\log X}{\log Y}.
$$
The following result is extracted from  \cite{S1}.

\begin{lem} \label{lem:4.1}
Assume that $u\leq 1$ and 
$$
 n>8\max\{3,1/u+1\}+1.
$$
There exists $\delta>0$ such that
$$
N(\mathcal{A};X,Y)=\mathfrak{S}\mathfrak{I}
X^{n-2}Y^{n-2}+O(X^{n-2-\delta}Y^{n-2}),
$$
where 
\begin{equation}\label{eq:sing-series}
\mathfrak{S}=\sum_{q=1}^{\infty}\sum_{\substack{a\bmod{q}\\ \gcd(a,q)=1}}q^{-2n}\sum_{\mathbf{x},\mathbf{y}\bmod{q}}e(aF(\mathbf{x};\mathbf{y})/q),
\end{equation}
and 
\begin{equation}\label{eq:sing-integral}
\mathfrak{I}=\int_{-\infty}^{\infty}\int_{[-1,1]^{2n}}e(-\theta F(\mathbf{x};\mathbf{y})) \mathrm{d}\mathbf{x}\mathrm{d}\mathbf{y}\mathrm{d}\theta.
\end{equation}
\end{lem}
\begin{proof}
This is a straightforward  modification of  \cite[Thm.~1.1]{S1} with the data  $R=1, n_1=n_2=n, d_1=d_2=2$ and 
$\mathfrak{B}_1=\mathfrak{B}_2=[-1,1]^n$.
Assuming that $X$ is smooth, 
it follows from \cite[Lemma 2.2]{S2} that 
$\dim V_i^*\leq n+1$ for $i=1,2$. The only thing that requires further comment is the restriction to the open subset $\mathcal{A}.$

On the minor arcs this restriction to an open set makes little difference. 
On the major arcs, however, we can extend the sum to all 
$\ZZ^{2n}$ by noting that there is a negligible contribution from
$(\x,\y)\in \ZZ^{2n}\setminus \mathcal{A}$. 
Indeed the total number of such vectors with $|\x|\leq X$ and 
$|\y|\leq Y$ is $O(X^{n-1}Y^n)$ since $\det(F_\x)$ and $\det(F_\y)$ are not identically zero. 
This completes the proof of the lemma.
\end{proof}

Next we consider the case where $X$ is substantially smaller than $Y$.  
As indicated above we shall proceed by fixing $\x$ and using Lemma \ref{lem:remove} to estimate the number of $\y$. 
Thus for fixed $\mathbf{x}$ let $N_{\mathbf{x}}(Y)$ be the number of integer vectors $|\mathbf{y}|\leq Y$ such that  $F(\x;\y)=0$.
We define
\begin{equation}\label{eq:tilde}
\widetilde N(X,Y)=\sum_{\substack{\mathbf{x}\in  \mathcal{A}_1\\ |\mathbf{x}|\leq X} }N_{\mathbf{x}}(Y),
\end{equation}
the estimation of which is the object of the following result.

\begin{lem}\label{lem:4.2}
 Let $X, Y\geq 1$ and suppose that 
$
n\geq 5.
$
Assume that $u<u_1$, where
 \begin{equation}\label{eq:u1}
u_1=\min \bigg\{\frac{n-2-\kappa_n}{3n+2+\kappa_n}, \frac{n-4+\kappa_n}{3n-\kappa_n}\bigg\}.
\end{equation}
Then  
for any $\eta<\frac{1}{4}$ there exists $\delta>0$ such that
\begin{align*}
\widetilde{N}(X,Y)=~&
\big(1+O(\eta)\big)
Y^{n-2}\sum_{\substack{\mathbf{x}\in  \mathcal{A}_1\\ |\mathbf{x}|\leq X}}\sigma_{\infty}(w_0;F_{\mathbf{x}})\mathfrak{S}(F_{\mathbf{x}})+O(X^{n-2-\delta}Y^{n-2}),
\end{align*}
where $\sigma_{\infty}(w_0;F_{\mathbf{x}})$ and $\mathfrak{S}(F_{\mathbf{x}})$ are given by Lemma \ref{lem:remove}.
\end{lem}

\begin{proof}
When  $\mathbf{x}\in  \mathcal{A}_1$ we note that $F_{\mathbf{x}}$ is a non-singular integral quadratic form with height $\| F_{\mathbf{x}} \|\ll |\mathbf{x}|^2$.
Applying Lemma \ref{lem:remove} in \eqref{eq:tilde}
we therefore deduce that 
$$
Y^{-(n-2)}
\widetilde{N}(X,Y)
=
\big(1+O(\eta)\big)
\sum_{\substack{\mathbf{x}\in  \mathcal{A}_1\\ |\mathbf{x}|\leq X}}\sigma_{\infty}(w_0;F_{\mathbf{x}})\mathfrak{S}(F_{\mathbf{x}})+O(\mathcal{E}),
$$
with
$$
\mathcal{E}
=X^{\frac{5n}{2}-1+\frac{\kappa_n}{2}}Y^{\frac{-n+2+\kappa_n}{2}+\ve}+X^{\frac{5n}{2}-2-\frac{\kappa_n}{2}}Y^{\frac{-n+4-\kappa_n}{2}+\ve},
$$
for any $\ve >0$. 
If $u_1$ is given by  \eqref{eq:u1} and $u=\log X/\log Y<u_1$, then it is not hard to  
 check that
$\mathcal{E}\ll  X^{n-2-\delta}$ for some $\delta>0$.
This completes the proof.
\end{proof}

We shall need to asymptotically evaluate the main term in Lemma \ref{lem:4.2}. This is the object of the following result. 

\begin{lem}\label{lem:constant}
Assume that 
$n>35$.
Then  for any $\eta<\frac{1}{4}$ there exists $\delta>0$ such that 
$$
  \sum_{\substack{\mathbf{x}\in  \mathcal{A}_1\\ |\mathbf{x}|\leq X}}
  \sigma_\infty(w_0;F_\mathbf{x})
    \mathfrak{S}(F_{\mathbf{x}})
 =\big(1+O(\eta)\big)
 \mathfrak{S}\mathfrak{I} X^{n-2}+O(X^{n-2-\delta}),
$$
where 
$\mathfrak{S}$
is given by  \eqref{eq:sing-series} and  $\mathfrak{I}$ is given by  \eqref{eq:sing-integral}.
\end{lem}

\begin{proof}
Let $X,Y\gg 1$ and put $u=\log X/\log Y$ are usual. We assume that 
$n\geq 5$ and $u<u_1$ in the notation of \eqref{eq:u1}. 
Then for any $\eta<\frac{1}{4}$ it follows from Lemma \ref{lem:4.2} that 
$$
\widetilde{N}(X,Y)
=\big(1+O(\eta)\big)
Y^{n-2}\sum_{\substack{\mathbf{x}\in  \mathcal{A}_1\\ |\mathbf{x}|\leq X}}\sigma_{\infty}(w_0;F_{\mathbf{x}})\mathfrak{S}(F_{\mathbf{x}})+O(X^{n-2-\delta}Y^{n-2}),
$$
for some $\delta>0$.
We may also treat the left hand side via an easy modification of the proof of 
Lemma \ref{lem:4.1}, in which 
the assumption on $n$ can be reduced to $n>8/u+9$, since $n>35$.
Assuming further  that  $u\leq 1$ this 
shows that 
$$
\widetilde{N}(X,Y)
=\mathfrak{S}\mathfrak{I}
X^{n-2}Y^{n-2}+O(X^{n-2-\delta}Y^{n-2}),
$$
for some $\delta>0$.
We want to compare these two estimates for $Y=X^{\frac{1}{u}}$, for a suitable parameter 
 $u\in [0,1]$ that satisfies $u<u_1$ and  $n>8/u+9$.
 But  $
 n>8/u_1+9$ if and only if $n>35$, and this is therefore 
 the condition under which the 
 asymptotic formula in the lemma is valid.
\end{proof}

We are now ready to deduce an estimate for $N(\mathcal{A};X,Y)$ which is valid for any $X$ and $Y$. 

\begin{lem}\label{lem:4.3}
Assume that  $n>35$.
Then for any $\eta<\frac{1}{4}$ 
there exists $\delta>0$ such that
$$
N(\mathcal{A};X,Y)=\sigma X^{n-2}Y^{n-2}+O(\eta X^{n-2}Y^{n-2})+O(X^{n-2-\delta}Y^{n-2}),
$$
where $\sigma=\mathfrak{S}\mathfrak{I}$.
\end{lem}

\begin{proof}
By symmetry we may assume that $X\leq Y$.
Let $u=\log X/\log Y\leq 1$
and recall the definition \eqref{eq:u1} of $u_1$.
The case  $u\geq u_1$ follows directly from  Lemma~\ref{lem:4.1},
 since one has $n>8\max\{3, 1/u_1+1\}+1$ if and only if $n>35$.

Suppose next that  $u< u_1$.
It follows from  Lemmas   \ref{lem:4.2}  and \ref{lem:constant} that 
\begin{align*}
\widetilde{N}(X,Y)=\sigma X^{n-2}Y^{n-2} +O(\eta X^{n-2}Y^{n-2})+O(X^{n-2-\delta}Y^{n-2}).
\end{align*}
It remains to  examine the difference
\begin{align*}
\widetilde{N}(X,Y)-N(\mathcal{A};X,Y)
&\leq 
\sum_{\substack{\mathbf{x}\in \mathbb{Z}^n\\ |\x|\leq X}}
\widetilde{N}_{\mathbf{x}}(Y),
\end{align*}
where 
  $\widetilde{N}_{\mathbf{x}}(Y)$ is the  number of  $\mathbf{y}\in 
  \ZZ^n\setminus \mathcal{A}_2$ such that $|\y|\leq Y$ and 
  $F(\x;\y)=0$.
 This is the number of integer vectors $|\y|\leq Y$ for which $F_\x(\y)=0$ and 
 $\Delta(\y)=0$, where  $F_\x$ is a non-singular  quadratic form and 
 $\Delta(\y)=\det(F_\y)$ is a non-zero form. These forms together define a complete intersection in $\PP^{n-1}$ of codimension $2$. Furthermore, 
 if it is reducible then it can't contain any linear components since the quadric $F_\x=0$ 
 doesn't contain  linear hyperplanes of codimension $1$.
Taking $m=n-3$ and $\ve=\frac{35}{36}-\frac{5}{3\sqrt{3}}$ in 
\cite[Cor.~2]{pila} it therefore follows that 
$\widetilde{N}_{\mathbf{x}}(Y)=O(Y^{n-3+\frac{2}{9}})$, where the implied constant only depends on $n$. This shows that 
$$
\widetilde{N}(X,Y)-N(\mathcal{A};X,Y)\ll X^{n}Y^{n-3+\frac{2}{9}}\ll X^{n-2-\delta}Y^{n-2}
$$
for some $\delta>0$, provided that $u<7/18$. But clearly $u_1\leq 1/3<7/18$.
The statement of the lemma now follows.
\end{proof}

We are now ready to complete the proof of Theorem \ref{t:1}.
We have 
\begin{align*}
N_U(B)=~&
\frac{1}{4}\#\left\{
(\x,\y)\in \mathcal{A}: 
\begin{array}{l}
\gcd(x_1,\dots,x_n)=\gcd(y_1,\dots,y_n)=1\\
|\x||\y|\leq B^{\frac{1}{n-2}},~
F(\x;\y)=0
\end{array}{}
\right\},
\end{align*}
on taking into account the action of the units $\{\pm 1\}$ on $\PP^{n-1}(\QQ)\times \PP^{n-1}(\QQ)$.
We take care of the coprimality conditions by using the M\"obius function. This leads to the expression
\begin{equation}\label{eq:finale}
N_U(B)=
\frac{1}{4}\hspace{-0.1cm}
\sum_{k_1,k_2=1}^\infty
\hspace{-0.1cm}
\mu(k_1)\mu(k_2)
\#\left\{
(\x,\y)\in \mathcal{A}: 
|\x||\y|\leq R,~
F(\x;\y)=0
\right\},
\end{equation}
where we have set 
$R=B^{\frac{1}{n-2}}/(k_1k_2)$.
Let $M(R)$ denote the inner cardinality.
We handle this by splitting into dyadic intervals and applying Lemma \ref{lem:4.3}.

Let
$
X_i=(1+\xi)^i$ and $Y_j=(1+\xi)^j$ for $i,j\geq 1$,
for a small parameter $\xi$. 
Then
$$
M(R)=\sum_{  (i,j)\in \mathcal{K}(R)}
\#\left\{ (\mathbf{x},\mathbf{y})\in {\mathcal{A}}:
           \begin{array}{ll}
            X_{i-1}<|\mathbf{x}|\leq X_i, ~
         Y_{j-1}<|\mathbf{y}|\leq Y_j  \\
             F(\mathbf{x};\mathbf{y})=0
           \end{array}
         \right\},
$$
where $\mathcal{K}(R)$ is the set of indices $(i,j)\in \NN^2$ such that 
$X_{i-1}Y_{j-1}\leq R$.
Recall the definition of $N(\mathcal{A};X,Y)$ from the start of \S \ref{s:4}.
Assuming that $n>35$, it now follows from Lemma \ref{lem:4.3} that 
\begin{align*}
M(R)
=~&\sum_{(i,j)\in \mathcal{K}(R)}
\sum_{\kappa_1,\kappa_2\in \{0,1\}} (-1)^{\kappa_1+\kappa_2}
N_{\mathcal{A}}(X_{i-\kappa_1},Y_{j-\kappa_2})\\
=~&
\sigma 
\sum_{ (i,j)\in \mathcal{K}(R)}
\left(1-2(1+\xi)^{-(n-2)}+(1+\xi)^{-2(n-2)}\right) (1+\xi)^{(i+j)(n-2)}\\
&+
\sum_{ (i,j)\in \mathcal{K}(R)}
O_{\xi}(\eta (1+\xi)^{(i+j)(n-2)})\\
&+\sum_{ (i,j)\in \mathcal{K}(R)}
O_{\xi}\left(
(1+\xi)^{(i+j)(n-2)}\min\{(1+\xi)^{i},(1+\xi)^{j}\}^{-\delta}\right)\\
=~& \Sigma+O_{\xi}(
\eta R^{n-2}\log R+
 R^{n-2}),
\end{align*}
say.
We note that 
\begin{align*} 
\Sigma
=~& \left(1-\frac{1}{(1+\xi)^{n-2}}\right)^2\sigma\sum_{ 1\leq i\leq \frac{\log R}{\log (1+\xi)}+2} (1+\xi)^{(n-2)i}
\hspace{-0.3cm}
\sum_{ 1\leq j\leq \frac{\log R}{\log (1+\xi)}+2-i} 
(1+\xi)^{(n-2)j}\\
=~&
\left(\frac{(1+\xi)^{n-2}-1}{(1+\xi)^{2(n-2)}}\right)
\sigma\sum_{ 1\leq i\leq \frac{\log R}{\log (1+\xi)}+2} \left(R^{n-2}-(1+\xi)^{i(n-2)}\right).
\end{align*}
Hence $\Sigma=cR^{n-2}\log R+O_{\xi}(R^{n-2})$, with 
\begin{align*}
c&=
\left(\frac{(1+\xi)^{n-2}-1}{(1+\xi)^{2(n-2)}}\right)
\frac{1}{\log (1+\xi)}\sigma\\
&=\frac{(n-2)\xi(1+O(\xi))}{\log (1+\xi)}\sigma\\
&=(n-2)\sigma+O(\xi),
\end{align*}
for small $\xi$.
Thus it follows that 
$$
\Sigma=(n-2)\sigma R^{n-2}\log R+O(\xi R^{n-2}\log R)+O_{\xi}(R^{n-2}).
$$

Returning to \eqref{eq:finale}
and executing the sum over $k_1$ and $k_2$,
we are finally led to the conclusion that 
$$
N_{U}(B)=\frac{1}{4}\frac{1}{\zeta(n-2)^2}\sigma B\log B+O(\xi B\log B)+O(\eta^{1/2} B\log B) +O_\xi(B),
$$
for any parameter $\xi>0$ and any $\eta\in (0,\frac{1}{4})$.
It is clear that this gives an asymptotic formula
$$
N_{U,H}(B)\sim \frac{1}{4}\frac{1}{\zeta(n-2)^2}\sigma B\log B,
$$
as $B\rightarrow \infty$.
Finally, we note that the leading constant exactly matches the prediction by Peyre \cite{P}, thanks to  an argument of   Schindler \cite[\S 3]{S2}.
This therefore concludes the proof of  Theorem \ref{t:1}.

\end{document}